\documentclass[11pt,a4paper,reqno,dvipsnames]{amsart}
\usepackage{geometry}
\usepackage{amsfonts}
\usepackage{amsmath}
\usepackage{amssymb,stmaryrd}
\usepackage{amsthm}
\usepackage{amsxtra}
\usepackage{bbm}
\usepackage{braket}
\usepackage{comment}
\usepackage{extarrows}
\usepackage{graphicx}
\usepackage{latexsym}
\usepackage{mathrsfs}
\usepackage{yhmath}
\usepackage{nicematrix}
\usepackage{mathtools}

\usepackage{float}
\usepackage{xcolor}
\usepackage{verbatim}
\usepackage[all,cmtip]{xy}
\usepackage{circledsteps}
\usepackage{listings}
\usepackage{relsize}
\usepackage[normalem]{ulem}

\usepackage[pdfborder={0 0 0}]{hyperref} 
\hypersetup{colorlinks, citecolor=blue, linkcolor=blue, urlcolor=blue}
\hypersetup{pagebackref}
\usepackage{bm}

\usepackage{tikz}
\usepackage{ytableau}

\usepackage[initials,lite]{amsrefs} 
\AtBeginDocument{%
    \def\MR#1{}
}

\usepackage{cleveref}
\Crefname{Lemma}{Lemma}{Lemmas}
\Crefname{Theorem}{Theorem}{Theorems}
\makeatletter
\def\cref@thmoptarg[#1]#2#3#4{%
    \ifhmode\unskip\unskip\par\fi%
    \normalfont%
    \trivlist%
    \let\thmheadnl\relax%
    \let\thm@swap\@gobble%
    \thm@notefont{\fontseries\mddefault\upshape}%
    \thm@headpunct{.}
    \thm@headsep 5\p@ plus\p@ minus\p@\relax%
    \thm@space@setup%
    #2
    \@topsep \thm@preskip               
    \@topsepadd \thm@postskip           
    \def\@tempa{#3}\ifx\@empty\@tempa%
      \def\@tempa{\@oparg{\@begintheorem{#4}{}}[]}%
    \else%
      \refstepcounter[#1]{#3}
      \@namedef{cref@#3@alias}{#1}
      \def\@tempa{\@oparg{\@begintheorem{#4}{\csname the#3\endcsname}}[]}%
    \fi%
    \@tempa}%
\makeatother

\theoremstyle{plain}

\newtheorem{Theorem}{Theorem}[section]
\newtheorem{Lemma}[Theorem]{Lemma}
\newtheorem{Corollary}[Theorem]{Corollary}
\newtheorem{Proposition}[Theorem]{Proposition}

\theoremstyle{definition}

\newtheorem{Assumptions and Discussion}[Theorem]{Assumptions and Discussion}

\newtheorem{Example}[Theorem]{Example}
\newtheorem{Definition}[Theorem]{Definition}
\newtheorem{Question}[Theorem]{Question}

\newtheorem{Remark}[Theorem]{Remark}

\newtheorem{Notation}[Theorem]{Notation}

\theoremstyle{remark}

\newtheorem*{acknowledgment*}{Acknowledgment}

\usepackage[shortlabels]{enumitem}
\SetEnumerateShortLabel{a}{\textup{(\alph*)}}
\SetEnumerateShortLabel{A}{\textup{(\Alph*)}}
\SetEnumerateShortLabel{1}{\textup{(\arabic*)}}
\SetEnumerateShortLabel{i}{\textup{(\roman*)}}
\SetEnumerateShortLabel{I}{\textup{(\Roman*)}}

\def\dim{\operatorname{dim}}

\def\Ht{\operatorname{ht}}

\def\ini{\operatorname{in}}

\def\ker{\operatorname{ker}}
\def\KK{{\mathbb K}}

\def\part{\operatorname{part}}

\newcommand\bda{{\bm a}}

\newcommand\bdb{{\bm b}}

\newcommand\bdc{{\bm c}}

\newcommand\bdd{{\bm d}}

\newcommand\bdX{{\bm X}}

\newcommand\bfX{\mathbf{X}}

\newcommand\bfS{\mathbb{S}}

\newcommand\calA{\mathcal{A}}
\newcommand\calB{\mathcal{B}}
\newcommand\calC{\mathcal{C}}

\newcommand\calF{\mathcal{F}}

\newcommand\calJ{\mathcal{J}}
\newcommand\calK{\mathcal{K}}
\newcommand\calL{\mathcal{L}}

\newcommand\calR{\mathcal{R}}

\newcommand\Spec{\operatorname{Spec}}

\newcommand{\rank}{\operatorname{rank}}
\def\reg{\operatorname{reg}}

\definecolor{MyGreen}{RGB}{34,136,51}

\newcommand{\Lr}{\ensuremath{\calL\times[r]}}

\begin{document}

\title[Gorenstein Special Fiber Rings of Ladder Determinantal Modules]{Gorenstein Special Fiber Rings of Ladder Determinantal Modules}

\author[ L. Fouli,  K.-N. Lin, H. Lindo,  M. Mostafazadehfard]{ Louiza Fouli,  Kuei-Nuan Lin, Haydee Lindo,  Maral Mostafazadehfard}

\thanks{2020 {\em Mathematics Subject Classification}.
    Primary 13C40, 
    13A30, 
    13H10; 
    Secondary 
    14M12, 
    14M15, 
    13C70. 
    }

\thanks{Keyword: Multi-Rees algebra, determinantal, Gorenstein, Hibi ring}

\address{Department of Mathematical Sciences, New Mexico State University, Las Cruces, NM 88003, USA}
\email{lfouli@nmsu.edu}

\address{Department of Mathematics, The Penn State University, McKeesport, PA, 15132, USA}
\email{linkn@psu.edu}

\address{Department of Mathematics, Harvey Mudd College, Claremont, CA, 91711, USA}
\email{hlindo@hmc.edu}

\address{Institute of Mathematics, Federal University of Rio de Janeiro, RJ,
21941-909, Brazil}
\email{maral@im.ufrj.br}

\begin{abstract}
A ladder determinantal module is an arbitrary direct sum of ideals of maximal minors of a generic ladder matrix. In this article, we give necessary and sufficient conditions for the special fiber ring of such modules to be Gorenstein. These conditions are expressed in terms of data obtained from the underlying matrix.
\end{abstract}

\maketitle

\section{Introduction}
Let $R$ be a Noetherian  ring or a standard graded ring  over a field $\KK$. For  a finitely  generated $R$-module $M$, the \emph{Rees algebra} of $M$, $\calR(M)$, is an algebra that naturally generalizes the notion of the Rees algebra of an ideal. It is defined as the quotient of the symmetric algebra of $M$ by its torsion.  Geometrically, $\calR(M)$ corresponds to the blowup of $\Spec(R)$ along $M$, which generalizes the notion of a blowup along an ideal. The Rees algebra of a module describes the closure of the graph of the rational map induced by $M$.
When $R$ is local with residue field $\KK$ or $R$ is a standard graded ring over a field $\KK$ one can consider another blowup algebra, namely, the special fiber ring of $M$, defined as $\calF(M)\coloneq \calR(M)\otimes_{R}\KK$. In other words, the special fiber ring of $M$ is the algebra generated by the images of the generators of $M$ over $\KK$. Geometrically,  $\calF(M)$
is the homogeneous coordinate ring of the exceptional fiber of the blowup of $\Spec(R)$ along $M$.

When a coordinate ring is Gorenstein, the associated algebraic variety has particularly nice duality properties, for example canonical sheaves become invertible, \cite[Chapter~3]{BHCMbook}.   This often implies that the variety has mild singularities and a well-behaved dualizing complex, properties that  are essential in birational geometry and the minimal model program \cite[Chapter~21]{EiComm}. 

The Gorenstein property of special fiber rings has been extensively studied in commutative algebra and algebraic geometry.  
Lima and Pérez \cite{LimaPerez2012} established necessary and sufficient conditions for the special fiber (fiber cone) of a Hilbert filtration to be Gorenstein, relating it to combinatorial invariants such as lengths and numbers of generators.  
Earlier, Heinzer, Kim, and Ulrich \cite{HeinzerKimUlrich2004} showed that in a Gorenstein ring $R$, under suitable hypotheses,  the associated graded ring of an ideal $I$  is Gorenstein if and only if $\calF(I)$ and $R/I$ are both Gorenstein.
Further, Kustin, Polini, and Ulrich \cite{KustinPoliniUlrich2015} analyzed special fiber rings of height-three Gorenstein ideals, demonstrating that these rings often inherit favorable duality and structural properties.

 Our main object of study is the special fiber of ladder determinantal modules. A ladder determinantal module is a module that is the direct sum of finitely many copies of the ideal of maximal minors of a generic ladder matrix, that is a matrix with zeros and indeterminates such that the nonzero entries form a ladder shape, see \Cref{sec:prelim} for more details.  Conca, Herzog, and Valla,  \cite{CHV96}, studied the Rees and special fiber algebras of rational normal scrolls and determined the defining equations of these algebras. Moreover, by studying the symmetry of the \(h\)-vector, they identified necessary and sufficient conditions under which the special fiber ring of these rational normal scrolls is Gorenstein.  Recent work has advanced our understanding of the special fiber ring of determinantal ideals: for ideals of $2\times 2$ minors of sparse \(2\times n\) matrices Celikbas et al.\ showed that the defining equations of the special fiber ring are the expected (Plücker-type) relations and used this to deduce normality, Cohen–Macaulayness, and many related invariants for the special fiber ring, \cite{CDFGLPS}.  
Ramkumar and Sammartano proved that for several families of \(2\)-determinantal ideals  the Rees algebra and the special fiber ring are Cohen–Macaulay and have quadratic (Koszul) presentations \cite{RS}.

Lin and Shen classified when the special fiber rings of secant varieties of rational normal scrolls are Gorenstein, \cite{LSSecantNormal}.  Later,
same authors determined the defining equations of the Rees algebra and the special fiber ring of ladder determinantal modules, \cite[Theorem~5.5]{LinShenLadder}. Recently, Costantini and coauthors calculated various algebraic invariants of the special fiber algebra of a ladder determinantal module in \cite{CFGLLLM}. 

It turns out that the initial algebra of the special fiber ring of a ladder determinantal module is a Hibi ring  \cite[Theorem~4.8]{LinShenLadder}.  Via \textsc{Sagbi} deformation, we prove that one can reduce the study of the Gorenstein property for $\calF(M)$ to the study of the Gorenstein property for the Hibi ring arising from the initial algebra of the special fiber ring of the ideal of maximal minors of the given ladder matrix, \Cref{r=1tor}.

It is well known that a Hibi ring is Gorenstein if and only if the distributive lattice associated to the ring has a pure join-irreducible poset, that is a poset where all the maximal chains have the same length, \cite{HibiDistLatt}. Recently, Miyazaki showed that the homogeneous coordinate ring of a Schubert variety is an algebra with straightening law generated by a distributive lattice and gave a combinatorial description of when a Schubert cycle is Gorenstein  using the join-irreducible poset of the Hibi ring, \cite{MiyazakiSchubert}. A Schubert cycle 
corresponds to maximal minors of a one-sided ladder matrix. In our work, we consider  two-sided ladder matrices and extend our study to special fiber rings of direct sums of ideals of maximal minors.

We now describe the main results of this work. In \Cref{sec:prelim}, we set up our notation and framework. In \Cref{sec:Poset}, we give an explicit description of the join-irreducible poset of the Hibi ring that is the initial algebra of the special fiber algebra of a ladder determinantal module, \Cref{cor: join irr}.
More precisely, we associate the ladder matrix with the join-irreducible poset of the Hibi ring, and one can read off the join-irreducible elements from a given ladder matrix, see \Cref{cor: join irr} and \Cref{examplePoset} for an instructive example.  
Building on this straightforward connection, in \Cref{thm: GorL} and \Cref{thm: GorM} we are able to provide necessary and sufficient conditions in terms of the matrix structure for  the special fiber ring of ladder determinantal modules to be Gorenstein. 
As an application of our results we obtain explicit calculations of the regularity, reduction number, and $a$-invariant of these algebras when they are Gorenstein, \Cref{cor: reg}. 
Finally, we consider the question whether these special fiber rings we considered are $F$-regular. When $\calF(M)$ is Gorenstein it follows that it is $F$-regular, \Cref{FRegular}. We conclude this work with an example showing that even without the Gorenstein property, one can still show that the special fiber ring of an ideal of maximal minors of a ladder matrix is $F$-regular, \Cref{ex: F-regular}.

We note here that the ideals we consider are ideals of maximal minors of a ladder matrix and they are in general different than ladder determinantal ideals considered for example in \cite{ConcaLadder}. Moreover, the special fiber rings we consider are not in general quotients of polynomial rings by determinantal ideals.

\section{Preliminaries and notations}
\label{sec:prelim}

We adopt the convention that if $p$ and $q$ are two positive integers with $p\le q$, the interval $[p,q]$ denotes the set $\{p,p+1,\dots,q-1,q\}$. When $p=1$, we denote the interval $[1,q]$ by $[q]$. 
We will use the notation in the following definition throughout the article.

\begin{Definition}
    \label{def:Ladder} 
    Let $\KK$ be a field and let $m,n,r$ be positive integers with $n<m$.  Write $[m]=S_1 \cup S_2 \cup \cdots \cup S_n$, where for each $i\in [n]$, $S_i\coloneqq [u_i, v_i]$ for some $u_i, v_i$ positive integers with $u_i < v_i$.  Without loss of generality, we may assume that $1=u_1<  u_2 <  \cdots <  u_n < m$ and $1< v_1<  v_2 < \cdots < v_n=m$.
     For each $i\in [n]$ and each $j\in S_i$, let $x_{i,j}$ be an indeterminate over the field $\KK$. Let $\bfS \coloneq \{S_1, \ldots, S_n\}$.

    \begin{enumerate}[a]
   \item    Let  $\bfX_{\bfS}$ be the $n\times m$ matrix whose $(i,j)$-th entry is $x_{i,j}$ when $j\in S_i$ and $0$ otherwise.  We say that $\bfX_{\bfS}$ is the \emph{ladder matrix} associated to $\bfS$. 
    \item Let $R\coloneq\KK[\bfX_{\bfS}]=\KK[x_{i,j}: i\in[n], j\in S_i]$ be a polynomial ring over $\KK$ and $L\coloneq I_n(\bfX_{\bfS})$ be the $R$-ideal generated by the $n\times n$ (maximal) minors of $\bfX_{\bfS}$. 
     \item Let  $\calL\coloneqq \{\bdc=(c_1, \ldots, c_n)\in S_1\times \cdots \times S_n : c_i<c_{i+1} \mbox{ for all } i\in [n-1]\}$. Notice that by the ladder shape of $\bdX_{\bfS}$ we have  $\det[\bdc]\neq 0$ for any $\bdc\in \calL$, where  $\det[\bdc]$ is the determinant of the $n\times n$ submatrix of $\bfX_{\bfS}$ whose columns are $c_1, \ldots, c_n$.  Therefore, the set $\{ \det[\bdc] : \bdc\in \calL\}$ is a minimal generating set of the ideal $L$.
     
     \item The module $M=\bigoplus_{i=1}^rL$ is called a \emph{ladder determinantal module} associated to the matrix $\bfX_{\bfS}$. Let $\calL\times [r]  \coloneqq \{(\bdc, j) : \bdc\in \calL, j\in [r]\}$ be the cartesian product of $\calL$ and $[r]$ and notice that the set $\{(\det[\bdc], j): (\bdc, j)\in \Lr\}$ is a minimal generating set for $M$.
     \item Let $\tau$ be the lexicographic order on the ring $R$ induced by the lexicographic order on the entries in $\bfX_{\bfS}$, that is $x_{i,j}>_{\tau}x_{l,k}$ if and only if $i<l$ or $i=l$ and $j<k$.
     \item Let $N=\bigoplus_{i=1}^r \ini_{\tau}(L)$ be the direct sum of the initial ideal of $L$ with respect to $\tau$.
     \item For all $i\in[n]$ and all $j\in [n-1]$, let $\Delta_i\coloneq v_i-u_i$, $\epsilon_j\coloneq u_{j+1}-u_j$, and  $\theta_j\coloneq v_{j+1}-v_j$. We adopt the convention that $\epsilon_n = \theta_0=\infty$.
\end{enumerate}
\end{Definition}

 We examine the blowup algebras associated with the module $M= \bigoplus_{j=1}^r L$. 
The \emph{Rees algebra}, $\mathcal{R}(M)$, and the \emph{special fiber ring}, $\mathcal{F}(M)$, of such a module $M$ are defined as follows: 
\begin{align*}
\mathcal{R}(M) &\coloneqq R[L t_1, \ldots, L t_r ] \\
& \, =\bigoplus_{a_i \in \mathbb{Z}_{\ge 0}} L^{\sum_{i=1}^ra_i}t_1^{a_1}\cdots t_r^{a_r} 
   \, \subseteq R[t_1,\ldots,t_r], \\ 
   \smallskip
\mathcal{F}(M) & \coloneqq \calR(M) \otimes_R \KK  \cong \KK[\det[\bdc]t_j : (\bdc,j)\in \Lr ].\\   
\end{align*}

One can realize the special fiber ring as a quotient of a polynomial ring in the following way.
Let $\{T_{\bdc,j}: (\bdc,j)\in \Lr\}$ be a new collection of indeterminates over $\KK$ and let  $T\coloneq \KK[\{T_{\bdc,j}: (\bdc,j)\in \Lr \}]$. We define the following surjective homomorphism: 

\begin{align*}
     \psi: T \longrightarrow \mathcal{F}(M),
\end{align*}
determined by $\psi(T_{\bdc,j})=\det[\bdc]t_j$ for each  $(\bdc,j)\in \Lr$.
Then $\calF(M)\cong T/\ker (\psi)$, and the
The ideal $\ker (\psi)$ is called the \emph{defining} or \emph{presentation ideal} of $\mathcal{F}(M)$.

In order to study the Gorenstein property of $\calF(M)$ we use the method of \textsc{Sagbi} degeneration. 
Recall that for a $\KK$-algebra  $A$, a set of elements $\calB\subseteq A$  is called a \emph{\textsc{Sagbi}} basis for $A$, with respect to a monomial order $>$, if $\ini_>(A) = \KK[\ini_>(\calB)]$. In the case of the ladder determinantal modules we are considering, Lin and Shen proved that the natural generators of $M$ form a \textsc{Sagbi} basis for $\calF(M)$ and used this \textsc{Sagbi} basis to determine the defining equations of $\calF(M)$.

\begin{Theorem}[{\cite[Theorem~4.8. 5.5, Corollary~4.10. 5.7]{LinShenLadder}}] \label{FiberLadder}
Let $M$ and $N$ be as in \Cref{def:Ladder}.
\begin{enumerate}[a]
    \item The set $
    \{\det[\bdc] t_i : (\bdc, i) \in \calL\times [r]\}$ is a \textsc{Sagbi} basis for $\calF(M)$ with respect to an extended monomial order $\tau'$, that is $\ini_{\tau'}(\calF(M))=\calF(N)$. Moreover, there exists a monomial  order $\omega$ on $T$ such that $\ini_{\omega}( \calK)=\ini_{\sigma} (\calJ)$, where $\calK, \calJ$ are the defining ideals of $\calF(M)$ and $\calF(N)$, respectively.

    \item    The $\KK$-algebra $\calF(N)$ is a Hibi ring and is isomorphic to $T/\calJ$ where 
    $$ \quad \calJ=(T_{\bda}T_{\bdb}-T_{\bda \wedge \bdb}T_{\bda \vee \bdb}: \bda, \bdb \in \calL\times[r]).$$

    \item  $\calF(M)$ and $\calF(N)$ are both Cohen--Macaulay normal domains.
\end{enumerate}
\end{Theorem}

We note here that $\calF(N)$ is a Hibi ring associated to the distributive lattice $\Lr$. Indeed, the set $\Lr$ admits a partial order defined by 
\[(\bda,j_1)=(a_1,\ldots,a_n,j_1)\ge (\bdb,j_2)=(b_1,\ldots, b_n,j_2)\]  if and only if 
\(a_i\ge b_i \text{ for all } i\in [n] \text{ and } j_1\ge j_2 ,\) or equivalently, $[(\bda,j_1)]_i\ge [(\bdb,j_2)]_i$ for all $i\in [n+1]$, where the notation $[-]_i$ is used to denote the $i$-th entry of the corresponding tuple.
 Under this partial ordering, $\calL \times [r]$ becomes a distributive lattice with meet and join defined by:
 \[  
 \bdc\wedge \bdd=\min\{\bdc, \bdd\} \ \mbox{ and } \bdc\vee \bdd=\max\{\bdc, \bdd\} \mbox{ for any } \bdc, \bdd\in \Lr,\]
 where the minimum and maximum of two tuples  are the vectors of the component wise minimum and maximum, respectively. In other words,  for  $\bdc, \bdd\in \Lr$ we have $[\bdc\wedge \bdd]_i=\min\{[\bdc]_i, [\bdd]_i\}$ and $[\bdc\vee \bdd]_i=\max\{[\bdc]_i, [\bdd]_i\}$ for all $i\in [n+1]$.

 We close this section with the following remark.

\begin{Remark}\label{shrinkmatrix}
For all $i \in [2,n]$, let $S_i' = S_i$ if $u_{i-1} < u_i$, and let $S_i' = [u_i + 1, v_i]$ whenever $u_{i-1} = u_i$.  
Let $\bfS' = \{S_1, S_2', \ldots, S_n'\}$, and let ${\bfX}_{\bfS'}$ be the ladder matrix corresponding to $\bfS'$.  
Notice that $\ini_{\tau}(I_n(\bfX_{\bfS})) = \ini_{\tau}(I_n(\bfX_{\bfS'}))$.  
Furthermore, since $\tau$ is a diagonal monomial order, we may assume that $u_i \le v_{i-1} + 1$ for all $i \in [2,n]$.  
Hence, 
we may assume without loss of generality that $u_i < u_{i+1}$ and $v_i < v_{i+1}$ as in  \Cref{def:Ladder}. If $u_i=v_i$ for some $i\in [n]$, then $L=x_{i,u_i}I_{n-1}(\bfX'_{\bfS})$, where $\bfX'_{\bfS}$ is the ladder matrix obtained from $\bfX_{\bfS}$ after removing row $i$. Note that $\calF(M) \cong \calF(\bigoplus_{i=1}^r I_{n-1}(\bfX'_{\bfS})$ and thus we may assume without loss of generality that $u_i<v_i$ for all $i\in [n]$ as in \Cref{def:Ladder}. Moreover, under our assumptions we cover the case of a generic matrix by considering $\bfX_{\bfS}$ with $\epsilon_i=1$ and $\theta_i=1$ for all $i\in [n-1]$. Finally, if $n=m$, then $\calF(M)$ is a polynomial ring and there is nothing to prove in this case. Therefore we may assume that $n<m$. 
\end{Remark}

\section{The join-irreducible poset}\label{sec:Poset}

Through this article we will assume the notation as in \Cref{def:Ladder}. We begin by introducing a list of tuples that we will use to form the collection of join-irreducible elements of the poset $\Lr$.

\begin{Definition} \label{def: join irr}
For each $i\in [n]$ and each $j \in [u_i+1, v_i]$ set
    \[\bda_{i,j}\coloneq (u_1, u_2, \ldots, u_{i-1}, j,\max\{j+1, u_{i+1}\}, \ldots, \max\{j+(n-i), u_n\}).\]
For all $k\in[2,r]$ define $$\bda_{n+1,k}\coloneq 
    (u_1,\ldots,u_n,k).$$  
\end{Definition}

\begin{Remark}
By construction,  each $\bda_{i,j}$ is determined by its $i$-th component $[\bda_{i,j}]_{i}$ and for fixed $i$ we have  $|\{\bda_{i,j}\}_{j \in [u_i+1,v_i]}| = \Delta_i$ and $\bda_{i,j} \in \calL$. Indeed, for any $s\in [0,n-i]$ if $$[\bda_{i,j}]_{i+s} = \max\{j+s, u_{i+s}\}=j+s,$$ then  
 $u_{i+s} \le j+ s \le v_i +s \le v_{i+s} $, where the last inequality follows by the ladder structure of the matrix. Therefore, $u_{i+s} \le [\bda_{i,j}]_{i+s} \le v_{i+s}$ implying $\bda_{i,j} \in \calL$.
\end{Remark}

\begin{Definition} Let $P$ be any poset and $\bda ,\bdb \in P$, with $\bdb < \bda$. We say that $\bda$ \emph{covers} $\bdb$ and write $\bdb\lessdot \bda $ when there is no $\bdc\in P$ with $\bdb<\bdc<\bda$. 
\end{Definition}

\begin{Remark}
   A join-irreducible element of a poset $P$ is any element that covers exactly one other element in $P$. Our goal is to determine the join-irreducible elements of $\calL$ and to understand the poset structure of the poset of join-irreducible elements of $\Lr$ and $\calL$. 
   Note that when restricting to the set of join-irreducible elements, a join-irreducible element may cover several other join-irreducible elements; see \Cref{lem: coverinP}.  
\end{Remark}

The following result establishes that the elements we defined in \Cref{def: join irr} are in fact the join-irreducible elements of $\calL$.

\begin{Proposition}\label{prop:join-irrL}
The set
\[ P_\calL := \{\bda_{i,j}: i\in[n], j\in[u_{i}+1, v_i]\} \]
is the set of join-irreducible elements of $\calL$. 
\end{Proposition}

\begin{proof}
Let $P$ be the set of join-irreducible elements of $\calL$. We claim that $P=P_{\calL}$. By  \cite[Corollary~4.24]{CFGLLLM} it follows that $|P|=\sum_{i=1}^{n}\Delta_i$ and since $|P_{\calL}|=\sum_{i=1}^{n}\Delta_i$, the result follows once we show $P_{\calL}\subseteq P$.

We claim that for any $\bdc=(c_1,\ldots,c_n)\in \calL$, if $\bdc < \bda_{i,j}$, then $[\bdc]_i<[\bda_{i,j}]_i=j$. Indeed, for any $t\in [1,i-1]$ we have $[\bdc]_t=c_t=u_t=[\bda_{i,j}]_t$. Since $\bdc<\bda_{i,j}$, there exists $s\ge i$ such that $[\bdc]_s=c_s<[\bda_{i,j}]_s$ and $[\bdc]_t=[\bda_{i,j}]_t$ for all $t\le s-1$. If $s>i$, then by the construction of $\calL$ we have $u_s\le [\bdc]_{s}<[\bda_{i,j}]_s=\max\{j+s-i,u_s\}=j+s-i$. Moreover, by \Cref{def: join irr} we have $[\bda_{i,j}]_{s-1}=\max\{j+s-i-1,u_{s-1}\}=j+s-i-1$. Then 
$$[\bdc]_{s-1}=[\bda_{i,j}]_{s-1}=j+s-i-1<[\bdc]_{s}<[\bda_{i,j}]_s=j+s-i$$ which is impossible. Therefore, $s=i$ and $[\bdc]_i<[\bda_{i,j}]_i=j$.

Now define 
\[
\bdb_{i,j}\coloneq (u_1,\ldots,u_{i-1},\,j-1,\,\max\{j+1, u_{i+1}\}, \ldots, \max\{j+(n-i), u_n\}).
\]
Then $\bda_{i,j}$, $\bdb_{i,j}\in \calL$ and $\bdb_{i,j}<\bda_{i,j}$. If $\bdc<\bda_{i,j}$, then by the argument above we have $[\bdc]_t=[\bda_{i,j}]_t=[\bdb_{i,j}]_{t}$ for all $t\in [i-1]$ and $[\bdc]_i<[\bda_{i,j}]_i=j$. Hence $[\bdc]_i\le [\bdb_{i,j}]_i$. Moreover, $[\bdc]_t\le [\bda_{i,j}]_t=[\bdb_{i,j}]_t$ for all 
$t\neq i$. In particular, $\bdc\le \bdb_{i,j}$ and therefore, $\bdb_{i,j}\lessdot \bda_{i,j}$
, that is $\bda_{i,j}$ is a join-irreducible element of $\calL$ as claimed. 
\end{proof}

\begin{Corollary}\label{cor: join irr}
    The join-irreducible elements of $\calL\times [r]$ are precisely the elements \[ P_{\Lr} := \{(\bda_{i,j},1): i\in[n], j\in[u_i+1, v_i]\}\cup \{\bda_{n+1,k}: k\in [2,r]\}.\]
\end{Corollary}
   \begin{proof}
        Let $P$ be the set of join-irreducible element of $\Lr$. By \cite[Corollary~4.24]{CFGLLLM} we have $|P|=|P_{\Lr}|$ and as before it suffices to show that $P_{\Lr}\subseteq P$. Let $(\bdc,k)\in \Lr$. If $k=1$, then $(\bdc,1)\in P$ if and only if $\bdc\in P_{\calL}$. If $k\ge 2$, then $(\bdc,k-1)\lessdot (\bdc,k)$ and $(\bdd, k)\lessdot (\bdc,k)$ for any  $\bdd,\bdc\in \calL$ with $\bdd\lessdot \bdc$. Therefore when $k\ge 2$, $(\bdc, k)\in P$ if and only if there does not exist $\bdd\in \calL$ with $\bdd \lessdot \bdc$. The latter means that $\bdc=(u_1, \ldots, u_n)$ and the result follows. 
       \end{proof}

Having determined the elements of $P_{\Lr}$ or equivalently the elements of $P_\calL$, we turn our attention to the poset structure of $P_{\calL}$.

\begin{Lemma} \label{cor: north east}
       Let $\bda_{i,j}, \bda_{k,l}\in P_{\calL}$. Then $\bda_{k,l}\le \bda_{i,j}$ if and only if $k\ge i$ and $l-j\le k-i$.
\end{Lemma} 
\begin{proof}
    Suppose $\bda_{k,l} \le \bda_{i,j}$. By \Cref{def: join irr}, we must have $k \ge i$
and $[\bda_{k,l}]_t \le [\bda_{i,j}]_t$ for all $t \in [n]$. 
In particular,
\[
u_k < l = [\bda_{k,l}]_k \le [\bda_{i,j}]_k = \max\{j + k-i,\, u_k\},
\]
which implies $l-j\le k-i$.

    Conversely, suppose $k \ge i$ and $l-j\le k-i$.  
For any $t < k$, we have 
$[\bda_{k,l}]_t = u_t \le [\bda_{i,j}]_t$.
For any $t \ge k$, we obtain
\[
[\bda_{k,l}]_t = \max\{\,l + (t-k),\, u_t\} \le  \max\{\,j + (t-i),\, u_t\} = [\bda_{i,j}]_t\]
and hence $\bda_{k,l} \le \bda_{i,j}$.
\end{proof}

For any element of $P_{\calL}$, we determine  which elements it covers.

\begin{Lemma} \label{lem: coverinP}
If $\bda_{i,j}\in P_{\calL}$, then the only elements $\bda_{i,j}$ covers are the elements in the set $$ \{\bda_{i,j-1}, \bda_{i+1,j+1}\},$$ which may be the empty set.
Moreover, 
$\bda_{i,j-1} \in P_{\calL} $ if and only if $ j \ge u_i+2$ and $\bda_{i+1, j+1} \in P_{\calL} $ if and only if $j \ge u_{i+1}$ and $i\neq n$.  
\end{Lemma}

\begin{proof}

First notice that since $\bda_{i,j}\in P_{\calL}$, then $i\in [n]$ and $j\in[u_i+1,v_i]$. Hence by \Cref{def: join irr} $\bda_{i,j-1}\in P_{\calL}$ if and only if $j-1\in [u_i+1, v_i]$ and  the latter is equivalent to $j\ge u_{i}+2$, since $j\le v_i$. Similarly, $\bda_{i+1, j+1}\in P_{\calL}$ if and only if $i+1\in [n]$ and $j+1\in [u_{i+1}+1, v_{i+1}]$. Since $i\in [n]$ and $j\le v_i<v_{i+1}$ we can conclude that $\bda_{i+1, j+1}\in P_{\calL}$ if and only if $i\neq n$ and $j\ge u_{i+1}$.
By \Cref{cor: north east}, if $\bda_{i,j-1}\in P_{\calL}$, then $\bda_{i,j-1}<\bda_{i,j}$ and similarly, if $\bda_{i+1, j+1}\in P_{\calL}$, then $\bda_{i+1, j+1}<\bda_{i,j}$.

Now suppose that $\bda_{i,j}$ covers $\bda_{g,h}\in P_{\calL}$. Since $\bda_{g,h}<\bda_{i,j}$, then $g\ge i$ and $h-j\leq g-i$, \Cref{cor: north east}. If $g=i$, then $h< j$,  since $\bda_{g,h}<\bda_{i,j}$. Thus $h \le j-1<j$ and since $\bda_{g,h}, \bda_{i,j}\in P_{\calL}$, then $\bda_{i,j-1}$ exists and $\bda_{g,h}\le \bda_{i,j-1}$, again by \Cref{cor: north east}. Therefore, $\bda_{g,h}=\bda_{i,j-1}$, since $\bda_{i,j}$ covers $\bda_{g,h}$. 

If $g>i$, then we claim that $\bda_{i+1,j+1}$ exists and $\bda_{g,h}= \bda_{i+1,j+1}$. If $\bda_{i+1,j+1}$ does not exist, then $j<u_i+1$ and thus $h<u_i+1+g-i\le u_{g}+1$. This means $\bda_{g,h}$ does not exist, a contradiction to our assumption that $\bda_{i,j}$ covers $\bda_{g,h}\in P_{\calL}$. Thus $\bda_{i+1,j+1}$ exists and since $h-j\le g-i=1+g-(i+1)$, it follows that $\bda_{g,h}\le \bda_{i+1,j+1}$, \Cref{cor: north east}. Hence $\bda_{g,h}=\bda_{i+1,j+1}$.

Therefore, $\bda_{i,j}$ covers an element of $P_{\calL}$ if and only if $\{\bda_{i,j-1}, \bda_{i+1,j+1}\}\neq \emptyset$. Moreover, $\bda_{i,j}$ covers only the join-irreducible elements in $\{\bda_{i,j-1}, \bda_{i+1,j+1}\}$, when the set is not empty. 
\end{proof}

 We remark here that \Cref{lem: coverinP} implies that if neither $\bda_{i,j-1}$ and $\bda_{i+1,j+1}$ exist, then $\bda_{i,j}$ does not cover any elements of $P_{\calL}$.
Our next goal is to understand the connected components of $P_{\Lr}$. We introduce the following notation. 

\begin{Notation}\label{not: disconnected partition}
    Consider the set $$\calC=\{i\in [n-1]: v_i<u_{i+1}\} \cup \{n\}=\{q_1, \ldots, q_{t-1}, q_t=n: q_{i}<q_{i+1} \text{ for all } i\in[n-1]\}.$$ 
Notice that $i\in \calC$ if and only if $u_{i+1}=v_i+1$ for all $i\in [n-1]$ by \Cref{shrinkmatrix}.
    We partition $[n]$ as follows \[[n]=[1, q_{1}]\cup [p_{2}, q_2]\cup \cdots \cup[p_{t}, q_t],\] where $p_{i+1}=q_{i}+1$ for all $i\in[t-1]$. 
    For simplicity let $p_1=1$.
Notice that with these assumptions $\bfX_{\bfS}$ is a block matrix with $t$ blocks $\bfX_i$, where $\bfX_i$ is a ladder matrix in the rows $p_i$ to $q_i$:
    $$\bfX_{\bfS}=\begin{bmatrix}
        \bfX_1 &&&\\
        &\bfX_2&&\\
        &&\ddots&\\
        &&&\bfX_t\\
    \end{bmatrix}.$$
\end{Notation}

We now determine the connected components of $P_{\Lr}$.

\begin{Proposition} \label{components}
 The lattice $P_{\Lr}$ is connected if and only if $r=1$ and $|\calC|=1$.
   In particular, $P_{\Lr}$ has $|\calC|$ connected components when $r=1$ and $|\calC| +1$ connected components when $r> 1$. 
\end{Proposition}

\begin{proof}
Recall that  $P_{\Lr} = \{(\bda_{i,j},1): i\in[n],  j\in[u_i+1, v_i]\}\cup \{\bda_{n+1,k}: k\in [2,r]\}$, \Cref{cor: join irr}.  Let $\calC=\{i\in [n-1]: v_i<u_{i+1}\} \cup \{n\}=\{q_1, \ldots, q_{t-1}, q_t=n: q_{i}<q_{i+1} \text{ for all } i\in[n-1]\}$ be as in \Cref{not: disconnected partition}. For every $s\in [t]$ let $A_s\coloneq\{(\bda_{i,j},1): i\in [p_s,q_s], j\in[u_i+1,v_i]\}$. 
Let $B=\{\bda_{n+1,k}:k\in[2,r]\}$, which exists if and only if $r\ge 2$. We claim that $A_1, \ldots, A_t, B$ are the connected components of $P_{\Lr}$ when $r\ge 2$. 
In fact, $A_s$ is an isomorphic copy of the poset corresponding to the $s$-th block of $X_{\mathbb{S}}$.

Note that $B\cup \bigcup_{i=1}^t A_i=P_{\Lr}$. 
Observe that $(\bda_{i,j},1)$ is incomparable to $\bda_{n+1,k}$ for any $i\in [n]$, $j\in [u_i+1,v_i]$ and $k\in [2,r]$. Moreover, $B$ is  a connected component of $P_{\Lr}$ since $\bda_{n+1,k}<\bda_{n+1,k+1}$ for all $k\in [2,r-1]$.

 Let $s\in [t]$ and let $i\in [p_s,q_s]$. By \Cref{cor: north east} the set $\{(\bda_{i,j},1): j\in [u_i+1,v_i]\}$ is connected. If $i<q_s$, then  
$v_i\ge u_{i+1}$. Hence $(\bda_{i,v_i},1)>(\bda_{i+1,v_i+1},1)$ by \Cref{lem: coverinP}. Therefore, $A_s$ is connected.

We claim that for any $s, s'\in [t]$ with $s\neq s'$ the elements of $A_s$ and $A_{s'}$ are incomparable. Indeed, let $(\bda_{i,j},1)\in A_s$ and $(\bda_{g,h},1)\in A_{s'}$ and without loss of generality assume $s'>s$. Then $g>i$ and $\bda_{g,h} \not> \bda_{i,j}$. Since 
$$j\le v_i\le v_{q_s}<u_{q_s+1}=u_{p_{s+1}} \le u_g<h,$$ 
it follows that $h-j>u_g-v_i\ge g-p_{s+1}+1+q_s-i=g-i$, as $p_{s+1}=q_s+1$, $u_{k}-u_l\ge k-l$ and $v_k-v_l\ge k-l$ for any $k>l$. Thus,  $\bda_{g,h}\not<\bda_{i,j}$ and hence $ (\bda_{i,j},1), (\bda_{g,h},1)$ are incomparable by \Cref{cor: north east} as claimed. Therefore, $A_1, \ldots, A_t, B$ are the connected components of $P_{\Lr}$. In the case $r=1$, one can show as above that the sets $A'_s=\{\bda_{i,j}: i\in [p_s, q_s] ,j\in [u_i+1,v_i] \}$ are the connected components of $P_{\calL}$ for all $s\in [t]$.
\end{proof}

In our final result of this section we determine the minimal and maximal elements of $P_{\calL}$.

\begin{Proposition}\label{lem:beginning and end}
 The minimal elements of $P_{\calL}$ are $$\{\bda_{i,u_i+1}: \epsilon_{i}>1\}$$  
    and the maximal elements of $P_{\calL}$ are 
    $$\{\bda_{i, v_i}: \theta_{i-1}>1\}.$$
\end{Proposition}

\begin{proof}
By \Cref{lem: coverinP},  $\bda_{i,j} \in P_{\calL}$ covers only the elements $\{ \bda_{i,j-1}, \bda_{i+1, j +1}\}$ in $P_{\calL}$, when they exist.  It follows that $\bda_{i,j}$ is minimal if and only if this set is empty, that is, when
 $j=u_i+1$ and $u_i+1<u_{i+1}$ or $i=n$. The latter is equivalent to $j=u_i+1$ and $\epsilon_i>1$, since $\epsilon_n>1$ by convention. 

Similarly, $\bda_{i,j}$ is maximal if and only if $\{ \bda_{i,j+1}, \bda_{i-1, j -1}\}=\emptyset$.
Note, since $\bda_{i,j} \in P_{\calL}$, then $\bda_{i,j+1}$ does not exist if and only if $j = v_i$. Moreover, $\bda_{i-1,v_i-1}$ does not exist if and only if either $i=1$ or $v_i-1\not\in[u_{i-1}+1, v_{i-1}]$.
The latter is equivalent to either $i=1$ or $v_i-1>v_{i-1}$. 
Therefore $\bda_{i,j}$ is maximal if and only if $j =v_i$ and  $\theta_{i-1} >1$, since by convention $\theta_0>1$. 
\end{proof}

\begin{Remark} \label{M}
One may associate several combinatorial objects with $P_{\calL}$. We first define a two-sided ladder matrix $\bfX_{P_{\calL}}$ of the same size as $\bfX_{\bfS}$ as follows: for any $i\in [n]$ and any $j\in [m]$ the $(i,j)$-th entry of $\bfX_{P_{\calL}}$ is $\bda_{i,j}$ if $j\in [u_i+1, v_i]$ and zero otherwise. The nonzero entries of $\bfX_{P_{\calL}}$ are precisely the join-irreducible elements of $P_{\calL}$. By \Cref{lem: coverinP}, $\bda_{i,j}$ covers only the elements in positions $(i,j-1)$ and $(i+1, j+1)$ provided they are nonzero. Thus the poset structure of $P_{\calL}$ can be identified directly from $\bfX_{P_{\calL}}$, and therefore directly from $\bfX_{\bfS}$. 
Note that the difference between the ladder structures of $\bfX_{\bfS}$ and $\bfX_{P_{\calL}}$ lie only in the $(i, u_i)$ entries; consequently,  $\bfX_{\bfS}$ and $\bfX_{P_{\calL}}$ share the same values of $\epsilon_i$ and $\theta_i$. One can also see that $P_{\calL}$ can be associated to a skew Young diagram. Recently, Gasanova and Nicklasson associated a skew Young diagram to the join-irreducible poset of certain Hibi rings arising from planar lattices to compute the $h$-vectors of these rings \cite{GasanovaNicklasson2024}.
\end{Remark}

\begin{Example}\label{examplePoset}
    Consider the partition $[1,13]=[1,5]\cup[3,7]\cup [4,8]\cup [9,11] \cup [10,13] $. 
    Then the ladder matrix $\bfX_{\bfS}$ is 
\[
\setlength{\arraycolsep}{6pt} 
\begin{array}{c}
\bfX_{\bfS}  \,\,\,\, =
\begin{bmatrix}
x_{1,1} & x_{1,2} & x_{1,3} & x_{1,4} & x_{1,5} & 0 & 0 & 0 & 0 & 0 &0 &0 &0 \\
0 & 0 & x_{2,3} & x_{2,4} & x_{2,5} & x_{2,6} & x_{2,7} & 0 & 0 & 0 &0 &0 &0\\
0 & 0 & 0 & x_{3,4} & x_{3,5} & x_{3,6} & x_{3,7} & x_{3,8} & 0 & 0&0 &0&0\\
0 & 0 & 0 & 0 & 0 & 0& 0 & 0 & x_{4,9} & x_{4,10} & x_{4,11}&0&0\\
0 & 0 & 0 & 0 & 0 & 0 & 0 & 0 & 0 &x_{5,10}& x_{5,11}& x_{5,12} & x_{5,13}  \\
\end{bmatrix}
\end{array}
\]
and its associated join-irreducible ladder matrix $\bfX_{P_{\calL}}$ is given by
\[
\setlength{\arraycolsep}{6.6pt} 
\begin{array}{c}
\bfX_{P_{\calL}} =
\begin{bmatrix}
\text{ } 0 & \bda_{1,2} & \bda_{1,3} & \bda_{1,4} & \bda_{1,5} & 0 & 0 & 0 & 0 &0&0&0&0\\
\text{ } 0 & 0 & 0 & \bda_{2,4} & \bda_{2,5} & \bda_{2,6} & \bda_{2,7} & 0 & 0&0&0&0&0 \\
\text{ } 0 & 0 & 0 &  0 & \bda_{3,5} & \bda_{3,6} & \bda_{3,7} & \bda_{3,8} &0 &0&0&0&0\\
\text{ } 0 & 0 & 0 & 0 & 0 &0 & 0 & 0& 0& \bda_{4,10}& \bda_{4,11}&0&0\\
\text{ } 0 & 0 & 0 & 0 & 0 &0 & 0 & 0 & 0 & & \bda_{5,11} &\bda_{5,12}& \bda_{5,13}
\end{bmatrix}.
\end{array}
\]

We write $\bda_{k,l}\rightarrow \bda_{i,j}$ if and only if $\bda_{k,l}<\bda_{i,j}$. Then the structure of $P_{\calL}$ can be depicted by the diagram    
\begin{figure}[H]
  \[   \xymatrixcolsep{1pc} \xymatrixrowsep{.05pc} P_{\calL}:  \xymatrix{
\bda_{1,2} \ar[r] & \bda_{1,3}  \ar[r]& \bda_{1,4}  \ar[r]& \bda_{1,5}&&\\
&& \bda_{2,4} \ar[r]  \ar[ul] & \bda_{2,5}  \ar[r] \ar[ul] & \bda_{2,6}  \ar[r] \ar[ul] & \bda_{2,7}   & &&&\\
&&& \bda_{3,5}  \ar[r]  \ar[ul] & \bda_{3,6}  \ar[r] \ar[ul] & \bda_{3,7}  \ar[r] \ar[ul] & \bda_{3,8} \ar[ul]  &  &&\\
&&&&&   &  &  & \bda_{4,10}  \ar[r] & \bda_{4,11}  \\
&&&&&   &  &  & & \bda_{5,11} \ar[ul] \ar[r] & \bda_{5,12}  \ar[ul] \ar[r] & \bda_{5,13} 
}
\]
\caption{The poset $P_\calL$ derived directly from matrices $\bfX_{\bfS}$ and $\bfX_{P_\calL}$.}
\end{figure}
The matrix $\bfX_{\bfS}$ has two blocks each one giving rise to a connected component of $P_{\calL}$.
\end{Example}

\section{The Gorenstein property}\label{sec:GR}

In \Cref{sec:Poset} we described the structure of $P_{\Lr}$ explicitly. Our goal in this section is to provide necessary and sufficient conditions for $P_{\Lr}$ to be pure or equivalently for $\calF(M)$ to be Gorenstein.

Recall that  a subset $\calA$ of a poset $P$ is called a
\emph{chain} of $P$ if $\calA$ is a totally ordered subset with respect to the induced order. In other
words, a chain is a subset $\calA = \{a_1,a_2,\ldots, a_k\}$ of $P$ with $a_1 < a_2 <\cdots < a_k$.
The \emph{length} of a chain $\calA$ is $|\calA|-1$. Let $\rank(P)$ denote the rank of $P$, which is the length of the longest chain of $P$. We also recall that for a standard graded algebra $A$ over a field $\KK$, the Hilbert function $H_A(t)$  counts the dimension of the $t$-th graded component of $A$. The Hilbert series is given by  
\[
   H_A(t) \;=\; \frac{h(t)}{(1-t)^d},
\]
where $d$ is the Krull dimension of $A$, and 
\[
   h(t) = h_0 + h_1 t + \cdots + h_s t^s
\]
is a polynomial with integer coefficients. The $h$-vector of $A$ is denoted by  
\[
   h(A) = (h_0, h_1, \ldots, h_s).
\]

We first show that it suffices to work with $\calF(\ini_{\tau}(L))$ which is a Hibi ring, \Cref{FiberLadder}.

\begin{Proposition}\label{r=1tor}
For $r\ge 2$, the following are equivalent:
      \begin{enumerate}[a]
          \item $ \calF(M)$ is Gorenstein;
          \item $\calF(N)$ is Gorenstein;
          \item $\calF(\ini_{\tau}(L))$ is Gorenstein and $\rank (P_{\calL})=r-2$;
          \item $ \calF(L)$ is Gorenstein and $\rank (P_{\calL})=r-2$.
      \end{enumerate}     
\end{Proposition}

\begin{proof}
Recall that a graded algebra is Gorenstein if and only if it is a Cohen-Macaulay domain and its $h$-vector is symmetric, see \cite[Corollary~4.4.6]{BHCMbook}. 

    By \Cref{FiberLadder} we know that $\ini_{\tau'}(\calF(M))\cong \calF(N)$, where $N=\bigoplus_{i=1}^{r}\ini_{\tau}(L)$.
    Then by \cite[Proposition 2.4]{CHV96}, the Hilbert functions of $\calF(M)$ and $\calF(N)$ coincide and hence  $\calF(M)$ and $\calF(N)$ have the same $h$-vector.   Therefore, $\calF(M)$ is Gorenstein if and only if $\calF(N)$ is Gorenstein, since $\calF(M)$ and $\calF(N)$ are Cohen-Macaulay domains, \Cref{FiberLadder}. 
    Moreover, by \Cref{FiberLadder}, $\calF(N)$ is the Hibi ring of the distributive lattice $\Lr$. Thus,  $\calF(N)$ is Gorenstein if and only if the poset of join-irreducible elements 
    $P_{\calL\times[r]}$ is pure, \cite{HibiDistLatt}. The latter holds if and only if all the maximal chains of $P_{\calL\times[r]}$ have the same length.
    
    Finally, that by the proof of  \Cref{components} the maximal chains of $P_{\calL\times[r]}$ are in one-to-one correspondence with the maximal chains of $P_{\calL}$ along with the chain given by $\{\bda_{n+1, k}\}_{k=2}^{r}$, which has length $r-2$. Now the result follows. 
 \end{proof}

\begin{Remark}
 Notice that, by \Cref{r=1tor}, even if $\calF(L)$ is Gorenstein, the algebra $\calF(M)$ can only be Gorenstein for one value of $r$. On the other hand, for a fixed $r$, the condition that $\calF(M)$ is Gorenstein imposes restrictions on the shape of the ladder matrix, as we will show in \Cref{thm: GorM}.
\end{Remark}

  In light of \Cref{r=1tor}, we can now focus on determining necessary and sufficient conditions for $ \calF(\ini_{\tau}(L))$ to be Gorenstein. In the next result we compute the length of any saturated chain between any two comparable elements in $P_{\calL}$.

  \begin{Lemma} \label{lem: chain length}
    Any saturated chain between two join-irreducible elements $\bda_{i,j}, \bda_{k,l}\in P_{\calL}$  with $\bda_{k,l}< \bda_{i,j}$ has length $2(k-i)+j-l$.
\end{Lemma}

\begin{proof} Let $\bda_{i,j}, \bda_{k,l}\in P_{\calL}$ such that $\bda_{k,l}< \bda_{i,j}$.
Then  by \Cref{cor: north east}, $k\ge i$. Notice that by \Cref{lem: coverinP}, we have $\bda_{i+1, j+1}<\bda_{i,j}$ and $\bda_{i,j-1}<\bda_{i,j}$ as long as $\bda_{i+1, j+1}$ and $\bda_{i,j-1}$ exist. Therefore, the difference in the index between two consecutive elements in a saturated chain is either $(1,1)$ or $(0,-1)$.
Since $\bda_{k,l}<\bda_{i,j}$ we must have \[(k,l) = (i,j)+a(1,1) + b(0,-1) \]  for some $a,b \in \mathbb{N}$.
The unique solution is $a = (k-i)$ and $b= (k-i) + (j-l)$. Thus the length of any such saturated chain from $\bda_{k,l}$ to $\bda_{i,j}$ is $a+b = 2(k-i)+ (j-l).$  \end{proof}

In order to determine the structure of saturated chains in $P_{\calL}$ we introduce the notion of consecutive minima and maxima.

\begin{Definition}
    Let $\bda_{i,u_i+1}, \bda_{k, u_k+1}$ be two minimal elements of $P_{\calL}$ as in \Cref{lem:beginning and end}. We say that $\bda_{i,u_i+1}$ and $\bda_{k,u_k+1}$ with $i<k$ are \emph{consecutive minima} there is no other minimal element $\bda_{t,u_t+1}$ of $P_{\calL}$ with $i<t<k$. In particular, if $i<k$ and $\bda_{i,u_i+1}$, $\bda_{k,u_k+1}$ are consecutive minima, then for all $s\in [i+1,k-1]$ we have $\epsilon_{s}=1$. Similarly, we can define \emph{consecutive maxima}.
\end{Definition}

With the definition of consecutive minima and maxima we are now able to show how maximal chains are related in $P_{\calL}$.

\begin{Lemma}\label{lem:connected}
    Suppose  $v_{i} \ge u_{i+1}$ for all $i \in [n-1]$. Then for any two consecutive minimal elements in $P_{\calL}$ there exist saturated chains to the same maximal element of $P_{\calL}$.   
    Similarly,  for any two consecutive maximal elements in $P_{\calL}$ there exist saturated  chains to the same minimal element of $P_{\calL}$.
\end{Lemma}

\begin{proof}
Let $\bda_{i,u_i+1}, \bda_{k, u_k+1}\in P_{\calL}$ be two consecutive minimal elements. Without loss of generality we may assume $i<k$ and that for all $s\in [i+1,k-1]$ we have $\epsilon_{s}=1$. 

Notice first that $\bda_{i, u_i+1}\le \bda_{i,v_i}$, by \Cref{cor: north east}. Moreover, since $\epsilon_s=1$ for all $s\in [i+1,k-1]$ we have that $u_{k}=u_{i+1}+(k-(i+1))=u_{i+1}+(k-i)-1$. Hence, $u_k+1\le v_i+k-i$, since $u_{i+1}\le v_i$. Therefore, $\bda_{k, u_k+1}<\bda_{i,v_i}$, by \Cref{cor: north east}. Finally, since $\bda_{i,v_i}$ is larger than both $\bda_{i,u_i+1}$, $\bda_{k, u_k+1}$, then there exists a maximum that is the common end of a chain from each one of the minimal elements. Symmetrically, for two consecutive maximal elements $\bda_{i,v_i}, \bda_{k, v_k}\in P_{\calL}$ one can show that $\bda_{k,u_k+1}\le \bda_{i,v_i}$ and $\bda_{k,u_k+1}\le \bda_{k, v_k}$. 
\end{proof}

We now give necessary and sufficient conditions for $\calF(L)$ to be Gorenstein, provided that $P_{\calL}$ is connected. Recall that $P_{\calL}$ is connected if and only if $|\calC|=1$, 
or equivalently, $v_{i} \ge u_{i+1}$ for all $i \in [n-1]$, \Cref{components}. 

\begin{Proposition} \label{Thm: Gor1}
 Suppose  $v_{i} \ge u_{i+1}$ for all $i \in [n-1]$. Then 
    $\calF(L)$ is Gorenstein if and only if $\bfX_{\bfS}$ satisfies the following conditions:
    \begin{enumerate}[a]
\item $u_k-u_i = 2(k-i)$
whenever $\epsilon_i, \epsilon_k >1$;

\item $v_k-v_i = 2(k-i)$ whenever  $\theta_{i-1}, \theta_{k-1} >1$. 
    \end{enumerate}
\end{Proposition}
\begin{proof}
Notice that by \Cref{lem:beginning and end} whenever $\epsilon_i, \epsilon_k >1$, then $\bda_{i,u_i+1}, \bda_{k,u_k+1}$ are minimal elements in $P_{\calL}$. Similarly, $\bda_{i,v_i}, \bda_{k,v_k}$ are maximal elements of $P_{\calL}$ whenever $\theta_{i-1}, \theta_{k-1}>1$.

 Suppose first that $\calF(L)$ is Gorenstein. Then 
 by the proof of \Cref{r=1tor} $\calF(\ini_{\tau}(L)$ is Gorenstein and thus $P_{\calL}$ is pure, \cite{HibiDistLatt}. Hence all the maximal chains in $P_{\calL}$ have the same length.
 Let  $\bda_{i, u_i+1}, \bda_{k,u_k+1}\in P_\calL$  be any two consecutive minimal elements. Then by \Cref{lem:connected}, there exists a maximal element $\bda_{h,v_h}$ with $\bda_{i, u_i+1}\le \bda_{h,v_h}$ and $ \bda_{k,u_k+1}\le \bda_{h,v_h}$. Without loss of generality we may assume that  $h\le i<k$. Then by \Cref{lem: chain length}, we have $$2(i-h)+v_h-(u_i+1)=2(k-h)+v_h-(u_k+1)$$ or equivalently $(u_k-u_i) = 2(k-i)$. Similarly, using \Cref{lem:connected}, for any two consecutive maximal elements $\bda_{i, v_i}, \bda_{k,v_k}\in P_\calL$  we have $(v_k-v_i) = 2(k-i)$.
 
 Now for any pair of  minima $\bda_{i, u_i+1}, \bda_{k,u_k+1}\in P_\calL$ we let $j_1,\ldots, j_s$ be the indices of  minimal elements between $i$ and $k$. Then we have $$(u_k-u_i) =u_k-u_{j_1}+u_{j_1}-u_{j_2}+\cdots+u_{j_s}-u_i=2[(k-j_1)+(j_1-j_2)+\cdots+(j_s-i)]=2(k-i).$$ Similarly, $(v_k-v_i) = 2(k-i)$ for every pair of maxima $\bda_{i, v_i}, \bda_{k,v_k}\in P_\calL$.
 
For the converse, suppose now that conditions (a) and (b) hold. Consider two maximal chains  from $\bda_{i,u_{i+1}}$ to $\bda_{g,v_g}$ and  from $\bda_{k,u_k+1}$ to $\bda_{h,v_h}$.  The lengths of these chains are equal when $2(i-g)+v_g-u_i-1=2(k-h)+v_h-u_k-1$, \Cref{lem: chain length}. This is equivalent to $$2(i-k)+v_g-v_h=2(g-h)+u_i-u_k,$$ which holds by conditions (a) and (b).
\end{proof}

\begin{Remark}\label{rem:sparse}
\Cref{Thm: Gor1} extends the result in \cite[Corollary~4.8]{CDFGLPS}, as we provide necessary and sufficient conditions for $\calF(L)$ to be Gorenstein, when $\bfX_{\bfS}$ is a sparse $2\times n$ matrix as a sparse matrix can be rearranged into a ladder matrix as shown in \cite[Corollary~4.8]{CDFGLPS}. 
If $\bfX$ is a generic $n \times m$ matrix, then $\calF(I_n(\bfX))$ is Gorenstein by the fact that 
\[
    \ini_{\tau}\!\left(I_n(\bfX_{\bfS})\right) = \ini_{\tau}\!\left(I_n(\bfX)\right),
\]
where $\bfX_{\bfS}$ is an $n\times m$ ladder matrix with $\epsilon_i=\theta_j=1$ for all $i\in [n-1], j\in [n]$ and $\epsilon_n, \theta_0>1$. Note since  $m>n$, we have $v_{i} \ge u_{i+1}$ for all $i \in [n-1]$.
By the proof of \Cref{r=1tor}, the algebra $\calF(I_n(\bfX))$ is Gorenstein if and only if $\calF(\ini_{\tau}(I_n(\bfX)))$ is Gorenstein.  
By the same reasoning, $\calF(\ini_{\tau}(I_n(\bfX)))$ is Gorenstein if and only if $\calF(I_n(\bfX_{\bfS}))$ is Gorenstein.  
The claim now follows by \Cref{Thm: Gor1} as the required assumptions are vacuous in the generic case.
\end{Remark}

As seen in the proof of \Cref{Thm: Gor1}, when $P_{\mathcal{L}}$ is connected, to prove  that $\calF(L)$ is Gorenstein it suffices to verify conditions (a) and (b) for consecutive minima and maxima, respectively. 

\begin{Corollary}\label{thm: Gor1ForEpsilon}
        Suppose $v_{i}\ge u_{i+1}$ for all $i\in [n-1]$.  Then 
    $\calF(L)$ is Gorenstein if and only if the matrix $\bfX_{\bfS}$ satisfies the following conditions:
    \begin{enumerate}[a]
        \item For any $i\in [n-1]$ if $\epsilon_i>1$, then $\epsilon_j=1$ for all $j\in [i+1, \epsilon_i-2]$ and $\epsilon_{i+\epsilon_i-1}>1$;

        \item For any $i\in [n-1]$, if $\theta_i>1$, then  $\theta_{j}=1$ for all $j\in [i-(\theta_i-2), i-1]$ and $\theta_{i-(\theta_i-1)}>1$.
    \end{enumerate}
\end{Corollary}

\begin{proof}
By the proof of \Cref{Thm: Gor1}, we have that $\calF(L)$ is Gorenstein if and only if $u_k-u_i=2(k-i)$ and $v_j-v_h=2(j-h)$, whenever $\bda_{i,u_i+1}, \bda_{k,u_k+1}$ and $\bda_{h,v_h}, \bda_{j,v_j}$ are consecutive minima and maxima, respectively. Without loss of generality we will assume $i<k$ and $h<j$.  Since $u_k - u_{i} = \sum_{s=i}^{k-1} \epsilon_s$, it follows that $\bda_{i,u_i+1}, \bda_{k,u_k+1}$ are consecutive minima if and only if $u_k-u_i=\epsilon_i+k-1-i$. Similarly, $\bda_{h,v_h}, \bda_{j,v_j}$ are consecutive maxima if and only if $v_j-v_h=\theta_{j-1}+j-h-1$. 

 One can see that for $i\in [n-1]$ condition (a) holds if and only if  $\bda_{i,u_i+1}, \bda_{k,u_k+1}$ are consecutive minima with $k=i+\epsilon_i-1$. This is equivalent to $u_k-u_i=\epsilon_i+k-1-i=2(k-i)$.

Similarly,  $j-1\in [0,n-1]$ satisfies condition (b) if and only if $\bda_{h,v_h}, \bda_{j,v_j}$ are consecutive maxima, where $h=j-\theta_{j-1}+1$. As above, this is equivalent to  $v_j-v_h=2(j-h)$.
\end{proof}

 \begin{Example}\label{ex:Gorenstein}
     Consider the ladder matrix $\bfX_{\bfS}$ given by the partition \[[1,15]=[1,8]\cup[4,9]\cup [5,10]\cup [7,14]\cup [9,15].\] We have $\epsilon_1=3$, $\epsilon_2=1$, $\epsilon_3=\epsilon_4=2$, $\theta_1=\theta_2=\theta_4=1$, and $\theta_3=4$. Hence by \Cref{thm: Gor1ForEpsilon}, $\calF(L)$ is Gorenstein. Note that verifying this property for this example with a computer algebra program is computationally expensive as the ideal $L$ has $1769$ generators. 
 \end{Example}

\begin{Remark}\label{rem: block matrix}
    Recall that by \Cref{components} the connected components of $P_{\calL}$ are determined by the set $\calC=\{i\in [n-1]: v_i<u_{i+1}\} \cup \{n\}=\{q_1<\cdots <q_{t-1}<q_t=n\}$. Moreover, $\bfX_{\bfS}$ is a block matrix with $t$ blocks $\bfX_i$, where $\bfX_i$ is a ladder matrix in the rows $p_i$ to $q_i$, where $[n]=\bigsqcup_{i=1}^{t}[p_i,q_i]$ and $p_{i+1}=q_{i}+1$ for all $i\in[t-1]$ and $p_1=1$.

    For each $s\in [t]$ we let $i_s\coloneq\min\{i\in [p_s,q_s]: \epsilon_i>1\}$. Note that $i_s$ is well defined since $\epsilon_{q_s}=u_{p_{s+1}}-u_{q_s} \ge u_{p_{s+1}}-v_{q_s}>0$.  Moreover, for $j=p_s$, we have $\theta_{j-1}=v_{j}-v_{j-1}\ge u_{j}-v_{j-1}>0$, since $j-1=q_{s-1}$.  Therefore, for every $s\in [t]$ it follows that $\bda_{i_s, u_{i_s}+1}$ is a minimal element in $P_{\calL}$ and $\bda_{p_s, v_{p_s}}$ is a maximal element in $P_{\calL}$, by \Cref{lem:beginning and end}.
\end{Remark}

\begin{Theorem}\label{thm: GorL}
       Adopt \Cref{not: disconnected partition}. Then 
    $\calF(L)$ is Gorenstein if and only if the matrix $\bfX_{\bfS}$ satisfies the following conditions: 
    \begin{enumerate}[a]
    \item For every $s\in [t]$,
    \begin{enumerate} [i]
    \item $(u_k-u_i) = 2(k-i)$ whenever $\epsilon_i, \epsilon_k>1$ with $i, k\in [p_s, q_s]$;
    \item $(v_k-v_i) = 2(k-i)$  whenever  $\theta_{i-1}, \theta_{k-1} >1$ with $i, k\in [p_s, q_s]$.
    \end{enumerate}

   \item    For any $s\in [t]$ we have $\Delta_{p_s}+i_s-p_s=\Delta_1+i_1-1$, where $i_s=\min\{i\in [p_s,q_s]: \epsilon_i>1\}$.
\end{enumerate}
\end{Theorem}

\begin{proof}
Let $s\in [t]$ and notice that for $i_s$ as in \Cref{rem: block matrix} we have that $\bda_{i_s,u_{i_s}+1}$ is a minimum and $\bda_{p_s, v_{p_s}}$ is a maximum. By \Cref{lem: coverinP}, there is a maximal chain in $P_{\calL}$ from $\bda_{i_s,u_{i_s}+1}$ to $\bda_{p_s,v_{p_s}}$. By \Cref{lem: chain length} this chain has length $2(i_s-p_s)+v_{p_s}-u_{p_s}-1=\Delta_{p_s}+i_s-p_s-1$, since $u_{i_s}=u_{p_s}+(i_s-p_s)$.

Now $\calF(\ini_{\tau}(L))$ is Gorenstein if and only if the lengths of all maximal chains are equal. The lengths of the maximal chains in each connected component are equal if and only if condition (a) holds, by \Cref{Thm: Gor1} and its proof. Moreover, in each connected component the common length of a maximal chain  is $\Delta_{p_s}+i_s-p_s-1$ and thus these are all equal if and only if condition (b) holds. 
\end{proof}

\begin{Theorem}\label{thm: GorM}
       Adopt \Cref{not: disconnected partition} and suppose $r\ge 2$. Then 
    $\calF(M)$ is Gorenstein if and only if for every $s\in [t]$ the matrix $\bfX_{\bfS}$ satisfies the following conditions: 
    \begin{enumerate}[a]
\item 
 $(u_k-u_i) = 2(k-i)$ whenever $\epsilon_i, \epsilon_k>1$ with $i, k\in [p_s, q_s]$;
\item 
$(v_k-v_i) = 2(k-i)$  whenever  $\theta_{i-1}, \theta_{k-1} >1$ with $i, k\in [p_s, q_s]$.
\item $r=\Delta_{p_s}+i_s
-p_s+1$, where $i_s=\min\{i\in [p_s,q_s]: \epsilon_i>1 \}$.
\end{enumerate}
\end{Theorem}

\begin{proof}
By \Cref{r=1tor}, $\calF(M)$ is Gorenstein if and only if $\calF(\ini_{\tau}(L))$ is Gorenstein and all the maximal chains have length $\rank P_{\calL}=r-2$. The conclusion now follows by \Cref{thm: GorL} and its proof.
\end{proof}

\begin{Corollary}
    Let $\bfX_{\bfS}$ be an $n\times m$ ladder matrix with $\epsilon_i=\theta_i=1$ for all $i\in [n-1]$.  Then $\calF(M)$ is Gorenstein if and only if $r=m$. In particular, if $\bfX$ is a generic $n\times m$ matrix, then $\calF(\oplus_{i=1}^rI_n(\bfX))$ is Gorenstein if and only if $r=m$.
\end{Corollary}

\begin{proof}
   Notice first that since $\epsilon_i=\theta_i=1$ for all $i\in [n-1]$, then $\Delta_1=m-n$. Moreover, $\bda_{n,n+1}$ is the only minimal element in $P_{\calL}$  and $\bda_{1,m-n+1}$ is the only maximal element of $P_{\calL}$. The conclusion for the first assertion follows by \Cref{thm: GorM} by noting that $i_1=n$.

Finally, notice that $\ini_{\tau}(L)=\ini_{\tau} (I_n(\bfX_{\bfS}))=\ini_{\tau}(I_n(\bfX))$. The second statement now follows by a similar argument as in \Cref{rem:sparse}.
\end{proof}

In the case $\calF(M)$ is Gorenstein, we can immediately compute the regularity and $a$-invariant of $\calF(M)$ as well as the reduction number $r(M)$ of $M$.

\begin{Corollary}\label{cor: reg} Suppose that $\KK$ is an infinite field.
 \begin{enumerate}[a]
     \item  If $\calF(M)$ is Gorenstein, then $\reg (\calF(M))=r(M)=\sum_{i=1}^{n}\Delta_i$ and $a(\calF(M))=-r$. 

 \item  If $\calF(L)$ is Gorenstein, then  $\reg (\calF(L))=r(L)=\sum_{i=2}^{n}\Delta_i-i+1$ and $a(\calF(L))=-\Delta_1-i$, where  $i=\min\{j: \epsilon_j>1\}$. 
  \end{enumerate}
\end{Corollary}

\begin{proof}
By \cite[Theorem 6.42]{HerHiOh}, \cite[Corollary 2.14, 4.24]{CFGLLLM}, and \cite[Proposition~2.6]{CFGLLLM}, we have  
\begin{align*}
\reg (\calF(M))&=r(M)=|P_{\calL\times [r]}|-\rank(P_{\calL \times [r]})-1 \quad \mbox{ and}\\
a(\calF(M))&=\reg (\calF(M))-\ell(M)
\end{align*}
By \cite[Corollary~4.24]{CFGLLLM} we have $|P_{\Lr}|=\sum_{i=1}^{n}\Delta_i+r-1$. If $r\ge 2$, then $\rank(P_{\calL \times [r]})=r-2$, by \Cref{thm: GorM}. If $r=1$, then $\rank(P_{\calL})=\Delta_1+i-2$, where  $i=\min\{j: \epsilon_j>1\}$ and the formulas for the regularity of $\calF(M)$ and the reduction number of $M$ follow. Finally, one can calculate the $a$-invariant using \cite[Theorem~3.4]{CFGLLLM}. 
\end{proof}

As another application of Proposition \ref{r=1tor} and Theorem \ref{thm: Gor1ForEpsilon}, we obtain the $F$-regularity of $\calF(M)$ in the case where this special fiber is Gorenstein. Recall that, $F$-regularity is a property of rings related to their behavior under the Frobenius morphism, and it plays a crucial role in understanding singularities in algebraic geometry. In general, $F$-regularity does not deform \cite{Si99}, but this holds when the ring is Gorenstein \cite[Corollary~4.7(c)]{HH94}; that is, if a ring $A$ is Gorenstein and $A / wA$ is $F$-regular for some nonzero divisor $w \in A$, then $A$ is $F$-regular.

\begin{Corollary}\label{FRegular}
  If $\bfX_{\bfS}$ satisfies the assumptions as in \Cref{thm: GorM},  then $\calF(M)$ is $F$-regular.
\end{Corollary}
\begin{proof}
    By \cite [Corollary 5.7]{LinShenLadder} the algebra $\calF(M)$ has rational singularities if $\mathrm{Char} \KK =0$ and is $F$-rational if $\mathrm{Char} \KK >0$. According to \cite[Corollary 4.7 (a)]{HH94} $F$-rationality and $F$-regularity are equivalent for Gorenstein rings. The conclusion follows by \Cref{thm: GorM}.
\end{proof}

Recall that by \cite[Theorem~4.15]{LinShenLadder} $\calF(L)$ is always $F$-rational. The next example shows that $\calF(L)$ can be $F$-regular even without being Gorenstein.

\begin{Example}\label{ex: F-regular}
    Consider the matrix  
$$\bfX =\begin{bmatrix} 
x_{1,1}  & x_{1,2} &x_{1,3} & x_{1,4} & x_{1,5} & 0\\
 0 & 0  & 0 &x_{2,4} &x_{2,5} & x_{2,6}
  \end{bmatrix}.$$

Let $L=I_2(\bfX)$ and consider the special fiber ring $\calF(L)$. According to \Cref{thm: Gor1ForEpsilon} the ring $\calF(L)$ is not Gorenstein. We claim that $\calF(
L)$ is $F$-regular. Our proof uses the Glassbrenner and Jacobian criterions.
Let $T\coloneq\mathbb{K}[T_{(i,j)} : (i,j)\in \calL]$ and recall that $\calF(L)\cong T/\calK$,  where $\calK$ is the defining ideal of $\calF(L)$ and can be described explicitly as in \cite[Theorem~5.5]{LinShenLadder}.

Let $S=T/\calK$. By \cite[Theorem 3.1]{Glass}, it suffices to find a homogeneous element $\alpha\in T$ $\alpha\in T=\mathbb{K}[T_{(i,j)} : (i,j)\in \calL]$ with $\alpha\notin \calK$ for which $S_{\alpha}$ is regular and $$\alpha(\calK^{[p^e]}):\calK)\not \subseteq \frak{m}^{[p]},$$ where $\frak{m}$ is the homogeneous maximal ideal of $T$ and $e$ is some positive integer. We claim that $\alpha=T_{(1,4)}$ satisfies the conditions above.

We first construct a complete intersection ideal $\frak{a}\subset \calK$ of height $\Ht(\calK)$. Let $\frak{a}=(f_1,\ldots, f_5)$, where 
 \begin{align*}
     f_1&=T_{(4,5)}T_{(3,6)}-T_{(3,5)}T_{(4,6)}+T_{(3,4)}T_{(5,6)},\\
     f_2&= T_{(4,5)}T_{(2,6)}-T_{(2,5)}T_{(4,6)}+T_{(2,4)}T_{(5,6)},\\
     f_3&=T_{(3,5)}T_{(2,6)}-T_{(2,5)}T_{(3,6)},\\
     f_4&=T_{(2,5)}T_{(1,6)}-T_{(1,5)}T_{(2,6)},\\
     f_5&=T_{(3,4)}T_{(1,5)}-T_{(1,4)}T_{(3,5)}.
 \end{align*}
 By \cite[Theorem 5.5]{LinShenLadder}, $f_i\in \calK$ for all $i\in [5]$. We will prove that $f_1,\ldots, f_5$ is a regular sequence of length $5$ by showing that the initial terms of those polynomials with respect to a term order $\succ$ are co-prime.  Let $\succ$ be the graded reverse lexicographic order on $T$ induced by the following order on the variables of $T$:
$$T_{(2,4)}>T_{(3,4)}>T_{(3,5)}>T_{(4,5)}>T_{(1,6)}>T_{(3,6)}>T_{(5,6)}>T_{(2,6)}>T_{(4,6)}>T_{(2,5)}>T_{(1,5)}>T_{(1,4)}.$$ 
 Then 
$\ini_{\succ}f_1=T_{(4,5)}T_{(3,6)}$, $\ini_{\succ}f_2=T_{(2,4)}T_{(5,6)}$, $\ini_{\succ}f_3=T_{(3,5)}T_{(2,6)},\ini_{\succ}f_4=T_{(2,5)}T_{(1,6)}$, $\ini_{\succ}f_5= T_{(3,4)}T_{(1,5)}$, and hence $f_1,\ldots, f_5$ is a regular sequence of length $5$. Note that $\dim T=|\calL|={6 \choose 2}-3=12$ and hence $\Ht(\calK)=5$, since $\dim (\calF(L))=\sum_{i=1}^2 \Delta_i+1=4+2+1=7$, by \cite[Theorem 3.4]{CFGLLLM}. Hence $\Ht(\frak{a})=\Ht(\calK)=5$.

We now show that  $T_{(1,4)} (\calK^{[p]}:\calK) \not \subseteq \frak{m}^{[p]} $. Notice that 
$$(\ini_{\succ}f_1)^{p-1} \cdots (\ini_{\succ}f_5)^{p-1}=\ini_{\succ}(f_1 \cdots f_5)^{p-1}\in \ini_{\succ}(\frak{a}^{[p]}:\frak{a}) \subseteq \ini_{\succ}(\calK^{[p]}:\calK),$$
where the last containment holds by \cite[Corollary 3.3]{PT}. Hence $$T_{(1,4)}(\ini_{\succ}f_1)^{p-1} \cdots (\ini_{\succ}f_5)^{p-1}\in T_{(1,4)}\ini_{\succ}(\calK^{[p]}:\calK).$$
By construction,  
$$T_{(1,4)}(\ini_{\succ}f_1)^{p-1} \cdots (\ini_{\succ}f_5)^{p-1} \not \in \frak{m}^{[p]}$$  and therefore, $T_{(1,4)}(\calK^{[p]}: \calK) \not \subseteq \frak{m}^{[p]}$.

It remains to show that $S_{\alpha}$ is regular, where $\alpha=T_{(1,4)}$. 
For this we apply the Jacobian Criterion, \cite[Corollary 16.20]{EiComm}. Let $J$ be the ideal generated by all the $5\times 5$ minors of the Jacobian matrix of $\calK$ and consider the following generators in the ideal $\calK$ that involve $T_{(1,4)}$:
\begin{align*}
    g_1&=T_{(1,4)}T_{(2,5)}-T_{(1,5)}T_{(2,4)}, \
    g_2=T_{(1,4)}T_{(3,5)}-T_{(1,5)}T_{(3,4)}, \ g_3=T_{(1,4)}T_{(2,6)}-T_{(1,6)}T_{(2,4)},\\ 
    g_4&=T_{(1,4)}T_{(3,6)}-T_{(1,6)}T_{(3,4)},\ g_5=T_{(1,4)}T_{(5,6)}-T_{(1,5)}T_{(4,6)}+T_{(1,6)}T_{(4,5)} .
\end{align*}
Notice that the matrix 
$${\begin{bmatrix}
    \frac{\partial g_1}{\partial T_{(2,5)}} &     \frac{\partial g_1}{\partial T_{(3,5)}}  &   \frac{\partial g_1}{\partial T_{(2,6)}}     &    \frac{\partial g_1}{\partial T_{(3,6)}}    &    \frac{\partial g_1}{\partial T_{(5,6)}}    \\
    \frac{\partial g_2}{\partial T_{(2,5)}}       & \frac{\partial g_2}{\partial T_{(3,5)}} &   \frac{\partial g_2}{\partial T_{(2,6)}}    &    \frac{\partial g_2}{\partial T_{(3,6)}}    &    \frac{\partial g_2}{\partial T_{(5,6)}}   \\
    \frac{\partial g_3}{\partial T_{(2,5)}}        & \frac{\partial g_3}{\partial T_{(3,5)}}       & \frac{\partial g_3}{\partial T_{(2,6)}} &    \frac{\partial g_3}{\partial T_{(3,6)}}   &   \frac{\partial g_3}{\partial T_{(5,6)}}     \\
    \frac{\partial g_4}{\partial T_{(2,5)}}        & \frac{\partial g_4}{\partial T_{(3,5)}}        & \frac{\partial g_4}{\partial T_{(2,6)}}         & \frac{\partial g_4}{\partial T_{(3,6)}} &   \frac{\partial g_4}{\partial T_{(5,6)}}     \\
    \frac{\partial g_5}{\partial T_{(2,5)}}        & \frac{\partial g_5}{\partial T_{(3,5)}}      & \frac{\partial g_5}{\partial T_{(2,6)}}        & \frac{\partial g_5}{\partial T_{(3,6)}}        & \frac{\partial g_5}{\partial T_{(5,6)}}\\
\end{bmatrix}}\small{=\begin{bmatrix}
    T_{(1,4)} &     0   &   0     &    0    &    0    \\
    0        & T_{(1,4)} &   0     &    0    &    0    \\
    0        & 0        & T_{(1,4)} &    0    &   0     \\
    0        & 0        & 0        & T_{(1,4)} &   0     \\
    0        & 0        & 0        & 0        & T_{(1,4)}\\

\end{bmatrix}}$$
is a submatrix of the Jacobian matrix of $\calK$ and hence $T_{(1,4)}\in \sqrt{J}$. Any prime ideal of $S_{\alpha}$ is of the form $P_{\alpha}$, where $P\in \Spec(S)$ with $\alpha\not \in P$. For any prime $P\in \Spec(S)$ with $T_{(1,4)}\not\in P$, we have $(S_{\alpha})_{P_{\alpha}}\cong S_{P}$.  By the Jacobian criterion \cite[Corollary 16.20]{EiComm} $S_{P}$ is a regular local ring if and only if $J\not \subset P$. Since $T_{(1,4)}\in \sqrt{J} \setminus P$, we have $S_{P}$ is regular and therefore, $S_{\alpha}$ is regular as claimed.
\end{Example}

The methods of \Cref{ex: F-regular} can not be extended in general. This leads to a natural question.
\begin{Question}
     Under what conditions is $\calF(L)$ $F$-regular?
\end{Question}

\vspace{0.5cm}

\begin{acknowledgment*}
We thank the reviewer for their careful reading of the manuscript and their thoughtful suggestions.

We thank New Mexico State University for hosting our group in the summer of 2025 and providing partial travel support. The travel to New Mexico State University was partially funded by the NSF grant DMS–2433082.
Kuei-Nuan Lin was partially supported by the AMS–Simons Research Enhancement Grants for Primarily Undergraduate Institution Faculty (2024).
Maral Mostafazadehfard was partially funded by CAPES–Brasil (Finance Code 001) and PROEX/-CAPES and benefited from APQ1-Faperj 211.012/2024 - SEI-260003/006427/2024.
Finally, Haydee Lindo was supported by the Simons Laufer Mathematical Sciences Institute.
\end{acknowledgment*}

\bibliography{ReesBib}

\end{document}